\documentclass[reqno,11pt]{amsart}
\usepackage{amsmath,amsthm,amscd,amssymb,graphicx,enumerate,latexsym}
\usepackage[raiselinks,colorlinks]{hyperref}
\hypersetup{citecolor=blue}

\numberwithin{equation}{section}

\theoremstyle{plain}
\newtheorem{thm}{Theorem}[section]
\newtheorem{lemma}{Lemma}[section]
\theoremstyle{definition}
\newtheorem{defn}{Definition}[section]
\theoremstyle{remark}
\newtheorem{remark}{Remark}[section]

\def\Re{\mathop{\rm Re}\nolimits}
\def\Im{\mathop{\rm Im}\nolimits}

\newcommand{\dist}{\text{\rm{dist}}}
\newcommand{\loc}{\text{\rm{loc}}}

\newcommand{\ac}{\text{\rm{ac}}}
\newcommand{\singc}{\text{\rm{sc}}}

\newcommand{\pp}{\text{\rm{pp}}}
\newcommand{\AC}{\text{\rm{AC}}}

\DeclareMathOperator*{\wlim}{w-lim}

\allowdisplaybreaks

\newcommand{\rel}[1]{\sim_{#1}}

\title[Wigner--von Neumann type potentials]{Schr\"odinger operators with slowly decaying Wigner--von Neumann type potentials}
\author{Milivoje Lukic}
\date\today
\email{milivoje.lukic@rice.edu}


\begin{document}

\begin{abstract}
We consider Schr\"odinger operators with potentials satisfying a generalized bounded variation condition at infinity and an $L^p$ decay condition. This class of potentials includes slowly decaying Wigner--von Neumann type potentials $\sin(ax)/x^b$ with $b>0$. We prove absence of singular continuous spectrum and show that embedded eigenvalues in the continuous spectrum can only take values from an explicit finite set. Conversely, we construct examples where such embedded eigenvalues are present, with exact asymptotics for the corresponding eigensolutions.
\end{abstract}

\maketitle
\section{Introduction}

In this paper, we will investigate a class of Schr\"odinger operators with decaying oscillatory potentials. The flagship example is the Wigner--von Neumann potential~\cite{WignerVonNeumann29} (see also \cite[Section XIII.13]{ReedSimon4}) on $(0,+\infty)$, which has asymptotic behavior
\begin{equation}
V(x) = - 8 \frac{\sin(2 x)}x + O(x^{-2}),\quad x\to\infty
\end{equation}
and the peculiar property that the Schr\"odinger operator $H=-\Delta+V$ has an eigenvalue at $+1$ embedded in the a.c.\ spectrum $[0,+\infty)$. In honor of this example, potentials of the form
\begin{equation}\label{1.2}
V(x) = \sum_{k=1}^K \lambda_k \frac{\cos(\alpha_k x +\xi_k)}{ x^{\gamma_k}}  + W(x), \quad \gamma_k>0, \quad W(x)\in L^1
\end{equation}
are often called Wigner--von Neumann type potentials. They have been the subject of much research, mostly restricted to $\gamma_k> \frac 12$; see Atkinson~\cite{Atkinson54}, Harris--Lutz~\cite{HarrisLutz75}, Reed--Simon~\cite[Thm XI.67]{ReedSimon3} and Ben-Artzi--Devinatz~\cite{Ben-ArtziDevinatz79}. It is proved there that for $\gamma_k>\frac 12$, $H$ has purely absolutely continuous spectrum on
 \[
 (0,\infty) \setminus \Bigl\{ \frac{ \alpha_k^2}{4} \Bigm\vert 1 \le k \le K\Bigr\}
 \]
 and that $E=\frac{\alpha_k^2}4$ can be in the pure point spectrum of $H$. Simon~\cite{Simon97}, following work by Naboko~\cite{Naboko86}, has even used Wigner--von Neumann potentials to construct decaying potentials with arbitrary positive pure point spectrum, including dense point spectrum.
 
 On a different note, Weidmann's theorem~\cite{Weidmann67} states that for $V=V_1+V_2$, where $V_1$ has bounded variation, $\lim\limits_{x\to\infty} V_1(x)=0$ and $V_2\in L^1(0,\infty)$, the Schr\"odinger operator $H=-\Delta+V$ has purely a.c.\ spectrum on $(0,+\infty)$. Inspired by those two results, one defines functions of generalized bounded variation. This class of functions includes functions of bounded variation and Wigner--von Neumann type potentials~\eqref{1.2} and is the natural class of potentials for the result that follows.

\begin{defn}\label{D1.1} A function $\beta\colon (0,+\infty)\to \mathbb{C}$ has \emph{rotated bounded variation} with phase $\phi$ if $e^{i \phi x} \beta(x)$ has bounded variation. A function $V\colon (0,+\infty)\to \mathbb{C}$ has \emph{generalized bounded variation} with the set of phases $A=\{\phi_1,\dots,\phi_L\}$ if it can be expressed as a sum
\begin{equation}\label{1.3}
V(x) = \sum_{l=1}^L \beta_l(x) + W(x)
\end{equation}
such that the $l$-th function $\beta_l$ has rotated bounded variation with phase $\phi_l$ and $W(x) \in L^1(0,+\infty)$.
\end{defn}

It is clear that a potential of the form \eqref{1.2} has generalized bounded variation with the set of phases $\{\pm \alpha_1, \dots, \pm \alpha_K\}$, since $\frac{\lambda_k}{x^{\gamma_k}} e^{\pm i[\alpha_k x + \xi_k]}$ has rotated bounded variation with phase $\mp \alpha_k$.

A real-valued function $V$ of generalized bounded variation obeys $V \in L^1_\loc(0,\infty)$, $V\in L^1(0,1)$ and
\begin{equation}
\lim_{n\to\infty} \int_n^{n+1} \lvert V(x) \rvert dx =0.
\end{equation}
Thus, $0$ is a regular point for $-\Delta+V$ and $-\Delta+V$ is limit point at $+\infty$. Therefore, by the general theory of one-dimensional Schr\"odinger operators (as described in \cite{ReedSimon2} or \cite{Teschl09}), the expression $-\Delta+V$ defines Schr\"odinger operators $H_\theta$ on $L^2(0,\infty)$, parametrized by $\theta\in [0,\pi)$. The domain of $H_\theta$ is
\begin{equation}\label{1.5}
D(H_\theta) = \{ u \in L^2(0,\infty) \mid  u, u' \in \AC_\loc,  -u''+ Vu \in L^2, u'(0) \sin \theta = u(0) \cos \theta  \}
\end{equation}
and $H_\theta: D(H_\theta) \to L^2(0,\infty)$ acts as
\begin{equation}\label{1.6}
(H_\theta u)(x)  = -u''(x) + V(x) u(x)
\end{equation}
The operator $H_\theta$ is self-adjoint, and for every $z \in \mathbb{C}$ with $\Im z > 0$, there is a solution of $- u''_z + V u_z = z u_z$ which is square-integrable near $\infty$. This is used to define the $m$-function
\[
m_\theta(z) =  \frac{u'_z(0) \cos\theta + u_z(0) \sin \theta }{u_z(0) \cos\theta - u'_z(0) \sin\theta}
\]
which, in turn, defines a canonical spectral measure $\mu_\theta$ by
\[
d \mu_\theta =\tfrac 1\pi \wlim_{\epsilon \downarrow 0}  m_\theta(x+i\epsilon) dx
\]
(the weak limit is with respect to continuous functions of compact support).
The importance of $\mu_\theta$ lies in the fact that the operator $H_\theta$ is unitarily equivalent to multiplication by $x$ on $L^2(\mathbb{R},d\mu_\theta(x))$.

Our first theorem describes the spectrum of operators with potentials of generalized bounded variation, with an $L^p$ condition on the decay. This is the analog of our results for orthogonal polynomials on the real line and on the unit circle (Theorems 1.1 and 1.2 from \cite{Lukic1}), and the proof will use ideas from \cite{Lukic1}. There is also closely related recent work for orthogonal polynomials and discrete Schr\"odinger operators by Wong~\cite{Wong09} and Janas--Simonov~\cite{JanasSimonov10}. For more on the history of this problem for those systems, see~\cite{Lukic1}.

\begin{thm}
\label{T1.1} Let $H_\theta$ be the Schr\"odinger operator given by \eqref{1.5} and \eqref{1.6}, where $V\colon (0,\infty)\to \mathbb{R}$ has generalized bounded variation with the set of phases $A$, and $V \in L^1+L^p$ for some positive integer $p$.
Then there is a finite set which depends only on $A$ and $p$,
\begin{equation}\label{1.7}
S_p =  \Bigl\{ \frac{\eta^2}4 \Bigm\vert \eta \in \bigcup_{k=1}^{p-1} (\underbrace{A+\dots+A}_{k\text{ times}}) \Bigr\},
\end{equation}
such that on $(0,\infty)\setminus S_p$, the spectral measure $\mu_\theta$ of $H_\theta$ is mutually absolutely continuous with Lebesgue measure. Thus,
\begin{enumerate}[(i)]
\item $\sigma_\ac(H_\theta) = [0,\infty)$;
\item $\sigma_\singc (H_\theta) = \emptyset$;
\item $\sigma_\pp (H_\theta) \cap (0,\infty) \subset S_p$ is a finite set.
\end{enumerate}
\end{thm}

As we increase $p$ in Theorem~\ref{T1.1}, we get larger sets $S_p$ of allowed positive eigenvalues in \eqref{1.7}. It is natural to ask whether these eigenvalues are really possible. We will construct examples for which these points are indeed eigenvalues of $H_\theta$.

For concreteness, we will construct examples with power-law decay. In what follows,  the potential will have the form
\begin{equation}\label{1.8}
V(x) = \sum_{k=1}^K \lambda_k \frac 1{x^{\gamma}} \cos(\alpha_k x + \xi_k(x)) + \beta_0(x), \qquad x\ge x_0
\end{equation}
where
\begin{equation}\label{1.9}
\gamma \in \left(\frac 1p, \frac 1{p-1} \right],
\end{equation}
$\lambda_k>0$ and
\begin{equation}\label{1.10}
\beta_0(x)\in C^1,\quad \beta'_0(x) = O(x^{-p\gamma}), \quad \beta_0(x) = O(x^{-\gamma}), \quad x\to \infty.
\end{equation}
The functions $\xi_k(x)\in C^1$ are chosen so that
\begin{equation}
\xi'_k(x) = O(x^{-(p-1)\gamma}), \quad x\to \infty,
\end{equation}
which ensures that \eqref{1.8} has generalized bounded variation with the set of phases $\{0,\pm \alpha_1,\dots,\pm \alpha_K\}$. Moreover, $\xi_k(x)$ will often obey the stronger condition
\begin{equation}\label{1.12}
\xi'_k(x) = O(x^{-p\gamma}), \quad x\to \infty,
\end{equation}
 implying in particular that $\xi_k(x)$ has bounded variation.

With the choice \eqref{1.9}, we have $V\in L^p$. Because of Theorem~\ref{T1.1}, we will focus on an eigenvalue $E \in S_p \setminus S_{p-1}$.
We will construct a real-valued solution $u(x)$ of
\begin{equation}\label{1.13}
-u''(x) + V(x)u(x)=Eu(x)
\end{equation}
with the asymptotic behavior
\begin{equation}\label{1.14}
\frac 1{\sqrt{E}} u'(x) + i u(x)  = A f(x) e^{i [\sqrt{E} x  + \theta_\infty]} (1+ o(1)), \quad x\to \infty
\end{equation}
with
\begin{equation}
f(x) = \begin{cases}
x^{-C \lambda_{j_1} \dots \lambda_{j_{p-1}}} & \gamma = \frac 1{p-1} \\
 \exp\left( - \frac C{1-(p-1)\gamma} \lambda_{j_1} \dots \lambda_{j_{p-1}} x^{1-(p-1)\gamma} \right)  & \gamma \in (\frac 1p, \frac 1{p-1})
 \end{cases}
\end{equation}
and $A, C>0$. At one step of the construction, we will have to cancel out a function of bounded variation, and the easiest way to do that will be by adjusting $\beta_0(x)$ in \eqref{1.8}; we state this as a theorem.

\begin{thm}\label{T1.2}
Let $V$ be given by \eqref{1.8} and \eqref{1.9} and let $E \in S_p \setminus S_{p-1}$. 
For any choice of $\alpha_1,\dots, \alpha_K$ away from an algebraic set of codimension $1$ (i.e., any choice not obeying a non-trivial polynomial relation), we can make a choice of $\beta_0(x)$ consistent with \eqref{1.10} and functions $\xi_k(x) \in C^1$ with \eqref{1.12} such that \eqref{1.13} has a real-valued solution $u(x)$ with asymptotics \eqref{1.14}.

The constant $C$ depends only on $\alpha_1$, \dots, $\alpha_K$ and $E$. In particular, for $\gamma \in (\frac 1p, \frac 1{p-1})$ or for large enough values of the product $\lambda_{j_1} \dots \lambda_{j_{p-1}}$, there is a choice of boundary condition $\theta$ such that $E$ is an eigenvalue of $H_\theta$.
\end{thm}

Adjusting $\beta_0(x)$ was just one way to construct solutions with the asymptotics \eqref{1.14}. Another approach is possible, which merely adjusts the $\lambda_k$; however, some restrictions on $E$ apply. We present the result of this method for the $p=3$ case; the general case can be analyzed in the same way. The significance of this theorem is that it doesn't require the addition of any new terms such as $\beta_0(x)$.

\begin{thm}\label{T1.3}
Let $V$ be given by \eqref{1.8} with $p=3$, \eqref{1.9} and $\beta_0(x)=0$ and let $E \in S_p \setminus S_{p-1}$. For any choice of $\alpha_1,\dots, \alpha_K$ away from an algebraic set of codimension $1$ (i.e., any choice not obeying a non-trivial polynomial relation), assume
\begin{equation}\label{1.16}
\min\left\{\frac{\alpha_1^2}4,\dots,\frac{\alpha_K^2}4\right\} < E < \max\left\{\frac{\alpha_1^2}4,\dots,\frac{\alpha_K^2}4\right\}.
\end{equation}
Then for any $\lambda_1, \dots, \lambda_K>0$ which obey
\begin{equation}
\sum_{k=1}^K \frac{\lambda_k^2}{4 E - \alpha_k^2} = 0,
\end{equation}
there exist functions $\xi_k(x) \in C^1$ with \eqref{1.12} such that \eqref{1.13} has a real-valued solution $u(x)$ with asymptotics \eqref{1.14}.

The constant $C$ depends only on $\alpha_1$, \dots, $\alpha_K$ and $E$. In particular, for $\gamma \in (\frac 1p, \frac 1{p-1})$ or for large enough values of the product $\lambda_{j_1} \dots \lambda_{j_{p-1}}$, there is a choice of boundary condition $\theta$ such that $E$ is an eigenvalue of $H_\theta$.
\end{thm}

\begin{remark}
The proofs of Theorems~\ref{T1.2} and \ref{T1.3} can be adapted to construct potentials with several embedded eigenvalues, as long as we have enough functions $\xi_k(x)$ at our disposal; we need to be able to separately control suitable linear combinations of the $\xi_k(x)$, given by \eqref{6.8}, which means that we can construct at most $K$ embedded eigenvalues. Choosing a subset of eigenvalues from $S_p \setminus S_{p-1}$ for which the linear combinations \eqref{6.8} are linearly independent, the construction just needs to be done simultaneously for all of them, which amounts to treating the equations for $\xi_k(x)$ as a coupled system of differential equations.
\end{remark}

The method used to construct solutions with the asymptotics \eqref{1.14} extends directly to the setting of orthogonal polynomials on the real line or unit circle, using the methods in Lukic~\cite{Lukic1}.

Independently, in the setting of discrete Schr\"odinger operators, Kr\"uger~\cite{Kruger12} has recently shown existence of embedded eigenvalues for  some potentials of generalized bounded variation. The method in \cite{Kruger12} is different than ours, but the constructed potentials are of the same form as in our Theorem~\ref{T1.2}. However, the precise asymptotics of the form \eqref{1.14}, the critical case $\gamma=\frac 1{p-1}$ and Theorem~\ref{T1.3} are new.

We begin by discussing some relevant properties of functions of generalized bounded variation in Section~\ref{S2} and Pr\"ufer variables in Section~\ref{S3}. Section~\ref{S4} will prove Theorem~\ref{T1.1}, except for some functional identities postponed to Section~\ref{S5}. Section~\ref{S6} builds upon the method of Sections~\ref{S4} and \ref{S5} to prove Theorems~\ref{T1.2} and \ref{T1.3}.

It is a pleasure to thank David Damanik, Yoram Last and Barry Simon for useful suggestions and discussions.

\section{Generalized bounded variation}\label{S2}

In this section we describe some properties of functions of rotated and generalized bounded variation which will be needed later. The first lemma lists some elementary properties.

\begin{lemma} Let $\phi, \psi\in\mathbb{R}$, let $A,B,C\subset \mathbb{R}$ be finite sets, and $\beta(x)$,  $\gamma(x)$ functions on $(0,\infty)$. Then
\begin{enumerate}[\rm (i)]
\item If $\beta(x)$ has rotated bounded variation, then $\beta(x)$ is bounded;

\item If $\beta(x)$ and $\gamma(x)$ have rotated bounded variation with phases $\phi$ and $\psi$, respectively, then $\beta(x)\gamma(x)$ has rotated bounded variation with phase $\phi+\psi$.
\end{enumerate}
\end{lemma}

The proof is straightforward and analogous to the proof of \cite[Lemma 2.1]{Lukic1}.

We remind the reader that our potential $V$ has the decomposition \eqref{1.3}, where $\beta_l$ has rotated bounded variation with phase $\phi_l\in A$ and $W \in L^1$. This decomposition is not unique, and it will be useful to make some adjustments to it. First, it will be useful to adjust the breakup in \eqref{1.3} so that $\beta_l \in C^1$.

\begin{lemma}\label{L2.2} If $V(x)$ is of the form \eqref{1.3}, with $\beta_l$, $W$ as described there, then the breakup can be adjusted so that, in addition to assumptions stated there, $\beta_l \in C^1$ and
\begin{equation}
\frac{d}{dx} \bigl( e^{i\phi_l x} \beta_l(x) \bigr) \in L^1
\end{equation}
\end{lemma}

This reduces to an observation made by Weidmann~\cite{Weidmann67} in conjuction with the proof of his theorem; therefore, we provide only an outline of the proof.

\begin{proof}[Outline of proof]
By linearity, it suffices to prove this fact for $L=1$ and by multiplying by $e^{-i\phi_l x}$, it suffices to prove it when $\beta = \beta_1$ has bounded variation. Since every function of bounded variation is the linear combination of four bounded real-valued increasing functions, by linearity it suffices to prove it when $\beta$ is an increasing function. 

Extend $\beta$ to a function on $\mathbb{R}$, with $\beta(y) = \lim\limits_{x\downarrow 0} \beta(x)$ for $y\le 0$. Pick $j \in C_0^\infty(-1,1)$ with $j\ge 0$ and $\lVert j\rVert_1=1$ and define $\tilde\beta = j * \beta$.  It can easily be verified that $\tilde\beta\in C^1$, $\tilde\beta$ is an increasing function, and $\lim\limits_{x\to \pm\infty} \tilde\beta(x) = \lim\limits_{x\to \pm\infty} \beta(x)$. Further, it can be proved that  since $\beta$ has bounded variation, $\beta - \tilde \beta \in L^1$. Thus, replacing $\beta$ by $\tilde\beta$ in the decomposition \eqref{1.3} and absorbing $\beta - \tilde \beta$ into $W(x)$ fulfills all the requirements.
\end{proof}

The next step is to show that $\beta_l \in L^p$.

\begin{lemma}\label{L2.3}
Let $f \in L^\infty$. If $f \in L^1 + L^p$, then $f \in L^p$.
\end{lemma}

\begin{proof}
Let $f\in L^\infty$ and
\begin{equation}\label{2.2}
f(x) = f_1(x) +f_p(x)
\end{equation}
 with $f_1 \in L^1$ and $f_p\in L^p$. Without increasing $\lVert f_1 \rVert_1$ and $\lVert f_p \rVert_p$, we can adjust the breakup \eqref{2.2} so that $\lvert f_1(x) \rvert \le  \lvert f(x) \rvert$ for all $x$ (for example, replace $f_1(x)$ by
 \[
 \tilde f_1(x) = \begin{cases} \frac{\lvert f_1(x) \rvert}{\lvert f_1(x) \rvert+\lvert f_p(x) \rvert} f(x) & f(x)\neq 0 \\
 0 & f(x) =0\end{cases}
 \]
 and replace $f_p$ by $\tilde f_p = f - \tilde f_1$). Now $f\in L^\infty$ implies $f_1 \in L^\infty$, and this together with $f_1\in L^1$ implies $f_1\in L^p$ by H\"older's inequality. Thus, $f =  f_1 + f_p \in L^p$.
\end{proof}

\begin{lemma}\label{L2.4}
Let $1\le p\le \infty$. If  $V\in L^1 + L^p$ and $V$ has generalized bounded variation with the set of phases $\{\phi_1,\dots, \phi_L\}$, with the $\beta_l$ as in Definition~\ref{D1.1}, then $\beta_l \in L^p$ for all $l$.
\end{lemma}

\begin{proof}
Since $\beta_k(x)$ has rotated bounded variation, it is a bounded function. Thus, $\sum_{k=1}^L\beta_k \in L^\infty$. Since $\sum_{k=1}^L \beta_k = V - W \in L^1 +L^p$, Lemma~\ref{L2.3} implies $\sum_{k=1}^L \beta_k \in L^p$.

At this point, it will be convenient to  introduce $\gamma_k(x) = e^{i\phi_k x} \beta_k(x)$. By Lemma~\ref{L2.2}, we can assume that $\gamma_k \in C^1$ and $\gamma_k' \in L^1$. Then for any $A>0$,
\begin{equation}\label{2.3}
\gamma_k(x) - \gamma_k(x+A) \in L^1
\end{equation}
because
\[
\int \lvert \gamma_k(x) - \gamma_k(x+A) \rvert dx \le \int \Bigl\lvert \int_x^{x+A} \gamma_k'(t) dt \Bigr\rvert dx \le A \int \lvert \gamma_k'(t) \rvert dt  =  A \lVert \gamma_k' \rVert_1 
\]
The idea now will be to take
\begin{equation}\label{2.4}
\sum_{k=1}^L e^{-i \phi_k x} \gamma_k(x)  \in L^p
\end{equation}
and translate it with various offsets, then take a linear combination of those translates in a way that will cancel all but one of the functions $\gamma_k$, up to an $L^1$ term. Fix $l$. Let $T = \{ \pi /(\phi_l - \phi_k) \mid 1 \le k \le L, k \neq l\}$.  Let $Q \subset T$ and denote $s(Q) = \sum_{q\in Q} q$. From \eqref{2.4}, perform the change of variables $x \mapsto x + s(Q)$ 
to get
\[
\sum_{k=1}^L e^{-i \phi_k (x+s(Q) )} \gamma_k(x+ s(Q))  \in L^p
\]
and then, by using \eqref{2.3}, 
\begin{equation}\label{2.5}
 \sum_{k=1}^L e^{-i \phi_k (x+s(Q))} \gamma_k(x)  \in L^1 + L^p
\end{equation}
Multiplying \eqref{2.5} by $e^{i \phi_l s(Q)}$ and summing over all $Q\subset T$ gives
\begin{equation}\label{2.6}
\sum_{Q\subset T} \sum_{k=1}^L e^{i (\phi_l - \phi_k) s(Q)} e^{-i \phi_k x} \gamma_k(x)  \in L^1 + L^p
\end{equation}
At this point, notice that
\[
\sum_{Q\subset T} e^{i (\phi_l - \phi_k) s(Q)} = \sum_{Q\subset T} \prod_{q\in Q} e^{i (\phi_l- \phi_k) q} = \prod_{q\in T} (1 + e^{i (\phi_l - \phi_k) q} ) = 2^{L-1} \delta_{kl}
\]
because of the choice of $T$. Thus, \eqref{2.6} is just $\gamma_l(x) \in L^1 + L^p$, which, because $\gamma_1 \in L^\infty$, implies by Lemma~\ref{L2.3} that $\gamma_l(x) \in L^p$.
\end{proof}

\section{Pr\"ufer variables}\label{S3}

Pr\" ufer variables were first introduced by Pr\"ufer \cite{Prufer26}. They are a tool for analyzing real-valued solutions of
\begin{equation}\label{3.1}
- u''(x) + V(x) u(x) = E u(x)
\end{equation}
and have found extensive use in spectral theory; see e.g.\ Kiselev--Last--Simon~\cite{KiselevLastSimon98}. For
\begin{equation}
E=\frac{\eta^2}4
\end{equation}
with $\eta>0$ and a real-valued nonzero solution $u(x)$ of \eqref{3.1},
we define modified Pr\"ufer variables by
\begin{align}
u'(x) & = \tfrac 12 \eta R_\eta(x) \cos (\tfrac 12 \eta x + \theta_\eta(x)) \label{3.3} \\
u(x) & = R_\eta(x) \sin (\tfrac 12 \eta x+ \theta_\eta(x)) \label{3.4}
\end{align}
We have departed from the usual notation by parametrizing in $\eta = 2\sqrt{E}$ rather than $k=\sqrt{E}$. We have also made a non-standard modification to include $\tfrac 12 \eta x$ in the $\cos$ and $\sin$ in \eqref{3.3}, \eqref{3.4}. With this change, if $V=0$ in some interval, then $\theta_\eta$ is constant in that interval by \eqref{3.5} below.

The $2\pi$ ambiguity in $\theta_\eta(x)$ is partly fixed by making $\theta_\eta(x)$ continuous in $x$; there is still a $2\pi$ ambiguity in $\theta_\eta(0)$, which won't matter to us. Substituting into \eqref{3.1}, we obtain a system of first-order differential equations for $\log R_\eta$ and $\theta_\eta$, which we'll write in terms of complex exponentials
\begin{align}
\frac{d\theta_\eta}{dx} & =  \frac{V(x)}{\eta} \bigl( \tfrac 12 e^{i[\eta x + 2\theta_\eta(x)]} + \tfrac 12 e^{-i[\eta x + 2\theta_\eta(x)]} - 1 \bigr)   \label{3.5} \\
\frac{d}{dx} \log R_\eta(x) & = \Im \Bigl( \frac {V(x)}{\eta} e^{i[\eta x + 2\theta_\eta(x)]} \Bigr) \label{3.6}
\end{align}
It will also be useful to think of $\theta_\eta$ and $\log R_\eta$ as parts of one complex function $\theta_\eta+ i \log R_\eta$, whose derivative is given by
\begin{equation}\label{3.7}
\frac{d\theta_\eta}{dx} +  i \frac{d}{dx} \log R_\eta(x) = \frac {V(x)}\eta \left( e^{i[\eta x + 2\theta_\eta(x)]} - 1 \right)
\end{equation}

Note that boundedness of $R_\eta(x)$ implies boundedness of the corresponding solution of \eqref{3.1}. Our potential $V$ is in $L^1+L^p$, so its negative part is uniformly locally $L^1$; thus, by Behncke~\cite{Behncke91} and Stolz~\cite{Stolz92}, boundedness of eigenfunctions allows one to use subordinacy theory of Gilbert--Pearson~\cite{GilbertPearson87} to imply purely absolutely continuous spectrum for the corresponding energies. We summarize this as a lemma.
 
\begin{lemma}\label{L3.1}
If for $E = \eta^2/4 \in S_\ac$,  $R_\eta(x)$ is bounded as $x\to \infty$ for any initial conditions $R_\eta(0)$, $\theta_\eta(0)$, then the spectral measure $\mu_\theta$ of the operator $H_\theta$ is mutually absolutely continuous on $S_\ac$ with the Lebesgue measure.
\end{lemma}

\section{Proof of Theorem 1.1}\label{S4}

In this section, we prove Theorem~\ref{T1.1}, except for some technical calculations deferred to the next section.
 
For a given set of phases $A$,  we define sets $A_p$ for $p\in \mathbb{N}$ by
\begin{equation}  \label{4.1}
A_p =\{0\} \cup  \bigcup_{k=1}^{p-1} (\underbrace{A+\dots+A}_{k\text{ times}})
\end{equation}
Since $A=-A$, the set $A_p$ contains all elements of
\[
(\underbrace{A+\dots+A}_{i\text{ times}}) - (\underbrace{A+\dots+A}_{j\text{ times}})
\]
for any $i\ge 1$, $j\ge 0$  and $i+j<p$.

\begin{defn}
Let $B\subset (0,+\infty)$ be a finite set. We define a binary relation $\rel{B}$ on the set of functions parametrized by $\eta\in (0,+\infty)$ by: $v_\eta(x) \rel{B}  w_\eta(x)$ if and only if
\begin{equation*}
\lim_{M\to +\infty} \int_0^M \bigl(v_\eta(x)-  w_\eta(x) \bigr) \, dx
\end{equation*}
converges uniformly (but not necessarily absolutely) in $\eta\in I$ for compact intervals $I\subset (0,+\infty)$ with $\dist(I,B)>0$.
\end{defn}

With this notation, if we are in the $L^p$ case, our goal will be to show that for any initial condition $R_\eta(0) = R^{(0)}>0$, $\theta_\eta(0) = \theta^{(0)} \in \mathbb{R}$, 
\begin{align}\label{4.2}
\frac d{dx} \log R_\eta(x) & \rel{A_p}  0
\end{align}
This implies boundedness of eigenfunctions with $E=\frac{\eta^2}4$, $\eta \notin A_p$, so by Lemma~\ref{L3.1}, it implies purely absolutely continuous spectrum on
\[
(0,\infty) \setminus \Bigl\{ \frac{\eta^2}4 \Bigm\vert \eta \in A_p \Bigr\}.
\]

The fact that convergence is uniform in $\eta$ is actually not needed, but will come automatically with the proof. Even more, the proof below actually shows that convergence is uniform in the initial condition $\theta^{(0)}$ as well.

In proving \eqref{4.2}, we will rely on the two recurrence equations \eqref{3.5}, \eqref{3.6}. 
Since $V$ has the decomposition \eqref{1.3}, by Lemma~\ref{L2.2} we assume that $\beta_l \in C^1$ and $\frac{d}{dx} \bigl( e^{i\phi_l x} \beta_l(x) \bigr) \in L^1$.

Starting from \eqref{3.6} and seeking to prove \eqref{4.2}, we are motivated to find a way to control expressions of the form $f(\eta) \Gamma(x) e^{i[\eta x + 2 \theta_\eta(x)]}$. The following lemma will give us a way of passing from expressions of the form $f(\eta) \Gamma(x) e^{i k [\eta x + 2 \theta_\eta(x)]}$, $k\in\mathbb{Z}$, to expressions with faster decay at infinity, but at the cost of a multiplicative factor with a possible singularity in $\eta$. These singularities will correspond to elements of $A_p$, which our method will have to avoid.

The idea behind this lemma is that for $\eta$ away from $\phi$, the exponential factor $e^{i \phi x}$ in this function helps average out parts of it when integrals are taken; this averaging is controlled by an integration by parts.

\begin{lemma}\label{L4.1}
Let $k\in\mathbb{Z}$ and $\phi\in\mathbb{R}$, with $k$ and $\phi$ not both equal to $0$. Let $B\subset \mathbb{R}$ be a finite set and $f\colon (0,+\infty) \setminus B \to \mathbb{C}$ be a continuous function such that 
\begin{equation}\label{4.3}
g(\eta)= - 2 k \frac{f(\eta)}{k \eta - \phi}
\end{equation}
is also continuous on $(0,+\infty)  \setminus B$ (removable singularities in $g$ are allowed).
\begin{enumerate}[(i)]
\item  If $\Gamma\in L^1(0,\infty)$, then
\begin{equation}
f(\eta) \Gamma(x) e^{k i[\eta x + 2\theta_\eta(x)]}  \rel{B}  0
\end{equation}
\item
 If $\Gamma\in C^1(0,\infty)$, $\frac{d}{dx} \bigl( e^{i\phi x} \Gamma(x) \bigr) \in L^1(0,\infty)$ and $\lim\limits_{x\to\infty} \Gamma(x) = 0$, then
\begin{equation}\label{4.5}
f(\eta) \Gamma(x)  e^{k i[\eta x + 2\theta_\eta(x)]}  \rel{B}  g(\eta) \Gamma(x)   e^{k i[\eta x + 2\theta_\eta(x)]} \frac{d\theta_\eta}{dx}
\end{equation}
\end{enumerate}
\end{lemma}

It might seem extraneous to explicitly require that both $f$ and $g$ be continuous; however, we want the lemma to cover both the case $k\neq 0$, when $f$ can be computed from \eqref{4.3} and is continuous if $g$ is, and the case $k=0$, $\phi\neq 0$, when $g\equiv 0$ and we want to allow $f$ to be any continuous function.

\begin{proof} (i) Since $\lvert e^{k i[\eta x + 2\theta_\eta(x)]} \rvert = 1$,
\[
\lim_{M\to \infty} \int_0^M f(\eta) \Gamma(x) e^{k i[\eta x + 2\theta_\eta(x)]} dx
\]
exists by dominated convergence and convergence is uniform since $f$ is bounded on compact subsets of $(0,+\infty)\setminus B$.

(ii) Let $\gamma(x) = e^{i\phi x} \Gamma(x)$ and $h(\eta) = f(\eta) / (k \eta - \phi)$. By the product rule,
\begin{align}
\frac d{dx} \Bigl[ h(\eta) \gamma(x)  e^{i[(k\eta - \phi) x + 2k \theta_\eta(x)]} \Bigr] & = h(\eta) \gamma'(x) e^{i[(k\eta - \phi) x + 2k \theta_\eta(x)]} \nonumber  \\
& \qquad\qquad + i h(\eta) \gamma(x) e^{i[(k\eta - \phi) x + 2k \theta_\eta(x)]} \Bigl[ k \eta - \phi + 2 k \frac{d\theta_\eta}{dx} \Bigr] \label{4.6}
\end{align}
Note that $h$ is continuous on $(0,+\infty)\setminus B$, by continuity of $g$ for $k\neq 0$ and by continuity of $f$ for $k=0$ and $\phi\neq 0$. Thus, $h$ is bounded on compact subsets of $(0,+\infty) \setminus B$ and together with $\lim\limits_{x\to \infty} \gamma(x) = 0$, this implies that $h(\eta) \gamma(x)  e^{i[(k\eta - \phi) x + 2k \theta_\eta(x)]}$ converges to $0$ uniformly in $\eta$ away from $B$ as $x\to \infty$.

Boundedness of $h$ away from $B$ together with $\gamma' \in L^1(0,\infty)$ implies
\[
h(\eta) \gamma'(x) e^{i[(k\eta - \phi) x + 2k \theta_\eta(x)]} \rel{B} 0
\]
Thus, taking the integral $\int_0^M dx$ of \eqref{4.6} and taking the limit as $M\to\infty$ gives
\[
h(\eta) \gamma(x) e^{i[(k\eta - \phi) x + 2k \theta_\eta(x)]} \Bigl[ k \eta - \phi + 2 k \frac{d\theta}{dx} \Bigr] \rel{B} 0
\]
which can be rewritten as \eqref{4.5} since $f(\eta) = (k\eta - \phi) h(\eta)$ and $g(\eta) = -2k h(\eta)$.
\end{proof}

To prove Theorem~\ref{T1.1}, we will need to apply Lemma~\ref{L4.1} iteratively, starting from \eqref{3.7}. After every application of Lemma~\ref{L4.1}, all the terms containing $W$ will be in $L^1$ and so $\rel{A_p} 0$, and we will be left with terms with products of $\beta_k$'s, with one more $\beta_k$ than we started with. We will repeat this procedure until we have products of $p$ of the $\beta_k$'s, at which point the $L^p$ condition completes the proof. Using also the form of \eqref{3.5}, we notice that we will only have terms of the form
\begin{equation}\label{4.7}
f_{I,K}(\eta;\phi_{j_1},\dots,\phi_{j_I}) \beta_{j_1}(x) \dots \beta_{j_I}(x) e^{iK[\eta x + 2\theta_\eta(x)]}
\end{equation}
with $I\ge 1$, $0\le K \le I$. Since terms of this form will occur with all permutations of $j_1,\dots, j_I$, we can agree to average in all of those terms, so that $f_{I,K}$ will be symmetric in $\phi_{j_1},\dots,\phi_{j_I}$.

When we apply Lemma~\ref{L4.1}(ii) to such a term, the appropriate $g_{I,K}$ will be
\begin{align}
g_{I,K}(\eta;\{\phi_i\}_{i=1}^I) & = - \frac{ 2 K  }{K \eta - \sum_{i=1}^I \phi_i } f_{I,K}(\eta;\{\phi_i\}_{i=1}^I) \label{4.8}
\end{align}
By \eqref{3.7}, we will start from
\[
 \frac {V(x)}\eta \left( e^{i[\eta x + 2\theta_\eta(x)]} - 1 \right)
\]
from which we read off the values of $f_{I,K}$ for $I=1$,
\begin{equation}\label{4.9}
f_{1,0}(\eta;\phi_1) = - \frac 1\eta, \qquad f_{1,1}(\eta;\phi_1) = \frac 1\eta
\end{equation}
By writing out which $g_{I-1,k}$ affect $f_{I,K}$ and remembering our convention to symmetrize in the $\phi_j$, we obtain a recurrence relation in $f_{I,K}$ and $g_{I,K}$,
\begin{align}
f_{I,K}(\eta;\{\phi_i\}_{i=1}^I)  & =  \frac 1\eta \sum_{k=K-1}^{K+1} \sum_{\sigma\in S_I} \frac 1{I!} \omega_{K-k} g_{I-1,k}(\eta;\{\phi_{\sigma(i)}\}_{i=1}^{I-1}), \quad I\ge 2 \label{4.10}	
\end{align}
where $\omega_a = \begin{cases} \tfrac 12 & a = \pm 1 \\ -1 & a =0 \end{cases}$.

There is one issue we haven't yet addressed: Lemma~\ref{L4.1}(ii) only applies when $k$ and $\phi$ aren't both equal to $0$. In our notation, this issue arises for terms
\[
f_{I,0}(\eta;\phi_{j_1},\dots,\phi_{j_I}) \beta_{j_1}(x) \dots \beta_{j_I}(x)
\]
with $\phi_{j_1} + \dots + \phi_{j_I}  = 0$. We will need a separate argument to eliminate these terms, and this will come from a symmetry property of $f_{I,0}$ proved in the next section.

Finally, we wish to prove that all iterations of Lemma~\ref{L4.1}(ii) can be performed with $B=A_p$, and for that we need to be able to control the singularities of $g_{I,K}$. This will come from a functional identity in terms of the $g_{I,K}$, also proved in the next section.

We will now present the proof, up to those technical calculations deferred to the next section.

\begin{proof}[Proof of Theorem~\ref{T1.1}]
As described above, we want to control $\frac d{dx} \log R_\eta(x)$ by means of an iterative process, and because of \eqref{3.7}, we start with
\begin{equation}\label{4.11}
 \frac {V(x)}\eta \left( e^{i[\eta x + 2\theta_\eta(x)]} - 1 \right)
 \end{equation}
which is a finite sum of terms of the form
\begin{equation}\label{4.12}
f_{I,K}(\eta;\phi_{j_1},\dots,\phi_{j_I}) \beta_{j_1}(x) \dots \beta_{j_I}(x) e^{i[K\eta  x + 2K\theta_\eta(x)]}
\end{equation}
(in fact, initially, only terms with $I=1$ are present). We then use Lemma~\ref{L4.1}(ii) to replace terms \eqref{4.12} by finite sums of terms of the same form, but with a greater value of $I$. We proceed with this process until we get terms with $I\ge p$; and by Lemma~\ref{L5.1}(iv), all terms with $I< p$ will have their corresponding $g_{I,K}$ continuous (and thus bounded) away from the set $A_p$. Terms \eqref{4.12} with $I\ge p$ are in $L^1$, so they are negligible in the relation $\rel{A_p}$.

Thus, the only terms we will be left with are the ones for which Lemma~\ref{L4.1}(ii) does not apply. These are terms with $K=0$ and $\phi_{j_1}+\dots+\phi_{j_I}=0$. However, for any such term
\begin{equation}\label{4.13}
f_{I,0}(\eta;\phi_{j_1},\dots,\phi_{j_I}) \beta_{j_1}(x) \dots \beta_{j_I}(x)
\end{equation}
in the sum, there is a corresponding term
\begin{equation}\label{4.14}
f_{I,0}(\eta;-\phi_{j_1},\dots,-\phi_{j_I}) \bar\beta_{j_1}(x) \dots \bar\beta_{j_I}(x)
\end{equation}
because we have chosen a decomposition \eqref{1.3} of $V$ such that for every $\beta_i$, there is a $\bar \beta_i$ in the decomposition. However, by Lemma~\ref{L5.1}(ii), the sum of \eqref{4.13} and \eqref{4.14} is purely real! Thus, when we take the imaginary part of \eqref{4.11}, by \eqref{3.7} we get
\[
\frac d{dx} \log R_\eta(x) \rel{A_p} 0
\]
which completes the proof.
\end{proof}

\section{Some functional identities}\label{S5}

In this section, we will establish some properties of the functions $f_{I,K}$ and $g_{I,K}$, which are used in the proof of Theorem~\ref{T1.1} to restrict the set of their nonremovable singularities and to prove the vanishing of terms which Lemma~\ref{L4.1} isn't able to handle.

We begin by establishing the notation. We will be dealing with functions of $1+n$ variables, where the first variable will be $\eta$ and the remaining $n$ will be phases. In applications these will be some of the phases of generalized bounded variation, but in this section we think of them merely as parameters of certain functions. We need a kind of symmetrized product for such functions:

\begin{defn}
For a function $p_{I}$ of $1+I$ variables and a function $q_{J}$ of $1+J$ variables, we define their symmetric product as a function $p_{I} \odot q_{J}$ of $1+(I+J)$ variables by
\begin{align*}
 (p_{I}\odot q_{J}) \bigl(\eta; \{\phi_i\}_{i=1}^{I+J} \bigr) & = \frac{1}{(I+J)!} \sum_{\sigma\in S_{I+J} } 
  p_{I} \bigl(\eta; \{\phi_{\sigma(i)}\}_{i=1}^{I} \bigr)  q_{J} \bigl(\eta; \{\phi_{\sigma(i)}\}_{i=I+1}^{I+J} \bigr) \end{align*}
 where $S_{I+J}$ is the symmetric group in $I+J$ elements.
\end{defn}

It is straightforward to see that $\odot$ is commutative and associative. We will also have a use for some auxiliary functions. Let $\Omega_a$, with $a\in\mathbb{Z}$, be a function of $1+1$ variables and let $\Xi_{I,K}$, for $0\le K\le I$, be a function of $1+I$ variables,
\begin{align}
\Omega_a(\eta;\phi_1) & = \begin{cases} 2 & a=0 \\
1 & a = \pm 1 \\
0 & \lvert a\rvert \ge 2
\end{cases} \\
\Xi_{I,K}(\eta;\{\phi_i\}_{i=1}^I) & = \begin{cases} 1 & I=1 \\ 0 & I\ge 2\end{cases}
\end{align}
These functions are, of course, constant but defining them as functions will be convenient for use with the symmetric product. We also introduce rescaled versions of the $f_{I,K}$ and $g_{I,K}$; for $I\ge 1$ and $0\le K \le I$, let
\begin{align}
F_{I,K} & = (-1)^{K-1} \eta^I  f_{I,K} \label{5.3} \\
G_{I,K} & = \frac{(-1)^{K}}2 \eta^I g_{I,K} \label{5.4}
\end{align}
We will also take the convention
\begin{equation}
F_{0,0} = G_{0,0} = 0
\end{equation}
Rescaling \eqref{4.8}, \eqref{4.9} and \eqref{4.10} gives
\begin{align}
F_{I,K} & = \Xi_{I,K} + \sum_{a=-1}^1 \Omega_a \odot G_{I-1,K+a} \label{5.6}  \\
G_{I,K}(\eta;\{\phi_i\}_{i=1}^I) & = \frac{ K }{K \eta - \sum_{i=1}^I \phi_i} F_{I,K}(\eta;\{\phi_i\}_{i=1}^I) \label{5.7}
\end{align}

Note that $F_{I,K}$ and $G_{I,K}$ have singularities, so we must be cautious when performing arithmetic with them. Note, however, that \eqref{5.6} and \eqref{5.7} define functions for complex values of all parameters, and that these functions are meromorphic in all parameters. Moreover, by \eqref{5.6} and \eqref{5.7}, $F_{I,K}$ and $G_{I,K}$ can only have singularities for parameters $\eta, \{\phi_i\}_{i=1}^I$ such that $k\eta = \sum_{i\in A} \phi_i$ for some $0< k<K$ and some $A\subset \{1,\dots,I\}$, which is only a finite set of hyperplanes in  $\mathbb{C}^{1+I}$. Thus, when proving identities like the ones that follow, we can perform the calculations for the case when all quantities are finite, and then extend by meromorphicity. We will do this without further explanation.

\begin{lemma}\phantomsection \label{L5.1}
\begin{enumerate}[\rm (i)]
\item For $0 \le K \le I$ and $0<k<K$, the identities
\begin{align}
F_{I,K} & = \sum_{i=0}^I F_{i,k} \odot G_{I-i,K-k} \label{5.8} \\
G_{I,K} & = \sum_{i=0}^I G_{i,k} \odot G_{I-i,K-k} \label{5.9}
\end{align}
hold for all values of parameters for which all terms occurring in both sides are finite; if seen as  equalities involving meromorphic functions, they hold identically.

\item If 
\begin{equation} \label{5.10}
\phi_1+\dots + \phi_I = 0
\end{equation}
then
\begin{equation} \label{5.11}
F_{I,0} (\eta, \phi_1,\dotsc, \phi_I) = F_{I,0} (\eta, -\phi_1,\dotsc, -\phi_I)
\end{equation}

\item Nonremovable singularities of $F_{I,K}$ and $f_{I,K}$ for $\eta>0$ are of the form
\begin{equation}\label{5.12}
\eta = \sum_{a=1}^b \phi_{m_a}
\end{equation}
with $b<I$.

\item Nonremovable singularities of $G_{I,K}$ and $g_{I,K}$ for $\eta>0$ are of the form \eqref{5.12} with $b\le I$.
\end{enumerate}
\end{lemma}

\begin{proof}
(i) We prove \eqref{5.8} and \eqref{5.9} simultaneously by induction on $I$. The statement is vacuous for $I\le 1$. Assume it holds for $I-1$. Then by \eqref{5.6},
\begin{align*}
 \sum_{i=0}^I F_{i,k} \odot G_{I-i,K-k} & =  \sum_{i=0}^I ( \Xi_{i,k} + \sum_{a=-1}^1 \Omega_a \odot G_{i-1,k+a}) \odot G_{I-i,K-k}
\end{align*}
Using the inductive assumption, we may apply \eqref{5.9} to the sums of $G\odot G$, unless $k+a\le 0$. But $k+a\le 0$ holds only for $k=1$, $a=-1$, and in this exceptional case $G_{i-1,k+a}=0$. Thus, 
\begin{align*}
 \sum_{i=0}^I F_{i,k} \odot G_{I-i,K-k} & =  \sum_{i=0}^I  \Xi_{i,k} \odot G_{I-i,K-k} + \sum_{a=-1}^1 \sum_{i=0}^I \Omega_a \odot G_{i-1,k+a} \odot G_{I-i,K-k} \\
 & =  \delta_{k-1} \Xi_{1,1} \odot G_{I-1,K-1} + \sum_{a=-1}^1 \Omega_a \odot ( G_{I-1,K+a} - \delta_{a+1} \delta_{k-1} G_{I-1,K-1} ) \\
 & = \delta_{k-1} \Xi_{1,1} \odot G_{I-1,K-1} + F_{I,K} - \Xi_{I,K} - \Omega_{-1} \delta_{k-1} G_{I-1,K-1} \\
  & = F_{I,K}
\end{align*}
where we used \eqref{5.6} in the third line and $\Xi_{1,1} = \Omega_{-1}$ and $\Xi_{I,K}=0$ (since $I\ge 2$) in the fourth. We have thus proved part of the inductive step, proving that \eqref{5.8} holds for our value of $I$. It remains to prove \eqref{5.9}.

By \eqref{5.7}, for any permutation $\sigma\in S_I$ we have
\[
\frac{K F_{I,K}(\eta;\{\phi_{\sigma(j)}\}_{j=1}^I)}{G_{I,K}(\eta;\{\phi_{\sigma(j)}\}_{j=1}^I)} = \frac{k F_{i,k}(\eta;\{\phi_{\sigma(j)}\}_{j=1}^i)}{G_{i,k}(\eta;\{\phi_{\sigma(j)}\}_{j=1}^i)} + \frac{(K-k) F_{I-i,K-k}(\eta;\{\phi_{\sigma(j)}\}_{j=i+1}^I)}{G_{I-i,K-k}(\eta;\{\phi_{\sigma(j)}\}_{j=i+1}^I)}
\]
Multiplying this by $G_{i,k}(\eta;\{\phi_{\sigma(j)}\}_{j=1}^i)  G_{I-i,K-k}(\eta;\{\phi_{\sigma(j)}\}_{j=i+1}^I)$, averaging in $\sigma\in S_I$ and summing in $i$ gives, by \eqref{5.8},
\[
\frac{K F_{I,K}}{G_{I,K}} \sum_{i=0}^I G_{i,k} \odot G_{I-i,K-k} = K F_{I,K}
\]
which gives \eqref{5.9}.

(ii) This identity will be obvious when written in the right way, but the notation is cumbersome. Let $A_I$ be the set of sequences $\overrightarrow{k} = (k_0, k_1, \dotsc, k_I)$ with $\lvert k_i - k_{i+1} \rvert \le 1$, $k_i\ge 1$ for $0<i<I$ and $k_0 = k_I = 0$, and let $H_{I,\overrightarrow{k}}$ be a function of $1+I$ variables given by
\begin{equation*}
H_{I,\overrightarrow{k},\sigma}(\eta;\phi_1,\dotsc,\phi_I) = \prod_{i=0}^{I-1} (2-\lvert k_{i+1}-k_i \rvert) \prod_{i=1}^{I-1}
\frac{k_i}{k_i \eta - \sum_{a=1}^i \phi_{\sigma(a)}}
\end{equation*}
This quantity is useful because, by a simple induction using \eqref{5.6} and \eqref{5.7},
\begin{equation}\label{5.13}
F_{I,0} = \frac 1{I!} \sum_{\sigma\in S_I} 
\sum_{\overrightarrow{k}\in A_I}  H_{I,\overrightarrow{k},\sigma}
\end{equation}
If $\overrightarrow{k}' = (k_I,k_{I-1},\dotsc, k_0)$ and $\sigma'$ is the ``reversed'' permutation from $\sigma$ defined by $\sigma'(j) = I+1 - \sigma(I+1-j)$, then \eqref{5.10} implies
\[
\frac{k'_i}{k'_i \eta + \sum_{j=1}^i \phi_{\sigma'(j)}}  = \frac{k_{I-i}}{k_{I-i} \eta - \sum_{j=1}^{I-i} \phi_{\sigma(j)}} 
\]
Taking the product $\prod_{i=1}^{I-1}$ of this, and similarly equating the other products, we obtain
\begin{equation*}
H_{I,\overrightarrow{k},\sigma} (\eta; \phi_1,\dotsc,\phi_I) = H_{I,\overrightarrow{k}',\sigma'}(\eta; -\phi_1,\dotsc,-\phi_I)
\end{equation*}
Summing in $\overrightarrow{k}$ and $\sigma$ and using \eqref{5.13} proves \eqref{5.11}.

(iii), (iv) We prove (iii) and (iv) simultaneously by induction on $I$.

If (iv) holds for $I<M$: by \eqref{5.6}, singularities of $F_{I,K}$ come from a $G_{I-1,k}$, so (iii) then holds for $I \le M$.

If (iii) holds for $I<M$: by applying \eqref{5.9} $K-1$ times, $G_{I,K}$ can be written as a sum of $K$-fold products of $G_{i,1}$ with $i\le I$, so all its nonremovable singularities are singularities of a $G_{i,1}$ with $i\le I$. By \eqref{5.7}, those can only be of the form \eqref{5.12} with $b=i\le I$, or coming from $f_{i,1}$, so again of that form with $b< i \le I$. Thus, (iv) holds for $I\le M$.

The statements for $f_{I,K}$ and $g_{I,K}$ follow from \eqref{5.3} and \eqref{5.4}.
\end{proof}

\section{Existence of embedded eigenvalues}\label{S6}

In previous sections, we used Lemma~\ref{L4.1} iteratively to prove $\log R_\eta \sim_{A_p} 0$, that is, to prove boundedness of solutions away from the set $A_p$. In this section, we will use the same approach at a point $\eta \in A_p$, looking for point spectrum. To establish what we are looking for, note the following simple lemma.

\begin{lemma}\label{L6.1}
Let $E>0$ and let $R(x)$, $\theta(x)$ be the Pr\"ufer variables corresponding to some solution of \eqref{3.1}, and assume that
\begin{equation}\label{6.1}
\frac d{dx} \log R(x) \sim - \frac B{x^{(p-1)\gamma}}
\end{equation}
and the limit $\theta_\infty = \lim\limits_{x\to\infty} \theta(x)$ exists. Then for some $A>0$,
\begin{equation}\label{6.2}
\frac 2\eta u'(x) + i u(x) = A f(x) e^{i [\tfrac \eta 2 x  + \theta_\infty]} (1+ o(1)), \quad x\to \infty.
\end{equation}
where
\begin{equation}
f(x) = \begin{cases}
x^{-B} & \gamma = \frac 1{p-1} \\
 \exp\left( - \frac B{1-(p-1)\gamma} x^{1-(p-1)\gamma} \right)  & \gamma \in (\frac 1p, \frac 1{p-1})
 \end{cases}
\end{equation}
\end{lemma}

\begin{proof}
If we define $r(x)$ by
\[
\frac d{dx} \log r(x) =  -\frac B{x^{(p-1)\gamma}},
\]
then $r(x) = \tilde A f(x)$. By definition, $\frac d{dx} \log R(x) \sim \frac d{dx} \log r(x)$ implies that $\frac{R(x)}{r(x)}$ has a finite non-zero limit, so we obtain $R(x) = A f(x) (1+o(1))$. The rest follows from \eqref{3.3} and \eqref{3.4}.
\end{proof}

\begin{proof}[Proof of Theorem~\ref{T1.2}] Our $V$ is of the form \eqref{1.8}, but we will also use the notation from Definition~\ref{D1.1}; namely, let $\{\phi_0,\dots,\phi_L\} = \{0,\pm \alpha_1,\dots,\pm \alpha_K\}$ and let $\beta_0(x), \dots,\beta_L(x)$ be the functions $\beta_0(x)$ and $\frac{\lambda_k}{x^\gamma} e^{\pm i (\alpha_k x + \xi_k(x))}$ with $1\le k\le K$. We focus on a point
\[
E = \frac{\eta^2}4 \in S_p \setminus S_{p-1}
\]
which means that $\eta$ is of the form
\begin{equation}\label{6.4}
\eta = \phi_{j_1} + \phi_{j_2} + \dots + \phi_{j_{p-1}}
\end{equation}
and that $\eta$ can't be similarly written as a sum of less than $p-1$ terms.

Note that unless the $\phi_j$ solve one of finitely many linear equations, $\eta$ can be represented in the form \eqref{6.4} in exactly one way. We will work under this assumption from now on.

With $\eta$ given by \eqref{6.4}, we start from \eqref{3.7} and apply Lemma~\ref{L4.1} iteratively. The process will go as in the proof of Theorem~\ref{T1.1}, except for the term
\begin{equation}\label{6.5}
f_{I,1}(\eta;\phi_{j_1},\dots,\phi_{j_{p-1}}) \beta_{j_1}(x) \dots \beta_{j_{p-1}}(x) e^{i[\eta  x + 2\theta(x)]}.
\end{equation}

By \eqref{6.4}, Lemma~\ref{L4.1}(ii) is not applicable to the term \eqref{6.5}. Remember that this term appears with all permutations of the set of indices, so denoting the number of distinct permutations of $(j_1,\dots, j_{p-1})$ by $C_1$, we obtain
\begin{equation}\label{6.6}
\frac d{dx} \log R(x) \sim \Im \left( \frac \Lambda{x^{(p-1)\gamma}} e^{i[\xi(x) + 2\theta(x)]} \right)
\end{equation}
where
\begin{equation}
\Lambda = C_1 f_{I,1}(\eta;\phi_{j_1},\dots,\phi_{j_{p-1}})  \lambda_{j_1} \dots \lambda_{j_{p-1}}
\end{equation}
and $\xi(x)\in \mathbb{R}$ is given by
\begin{equation}\label{6.8}
\xi(x) = \xi_{j_1}(x) + \dots + \xi_{j_{p-1}}(x).
\end{equation}
Conversely, once we construct appropriate $\xi(x)$, we will pick $\xi_j(x)$ obeying \eqref{6.8} by taking
\begin{equation}\label{6.9}
\xi_j(x) = c_j \xi(x)
\end{equation}
for some $c_j\in \mathbb{R}$ with $c_{j_1} + \dots + c_{j_{p-1}}=1$.

Note that $f_{I,1}$ is a rational function in $\eta, \phi_{j_1},\dots,\phi_{j_{p-1}}$; moreover, $(-1)^I f_{I,1}$ is strictly positive for large enough $\eta$, which follows by induction from the defining recurrence relations \eqref{4.10} and \eqref{4.8}. Thus, $f_{I,1}$ is a non-trivial rational function. We will assume from now on that 
\[
f_{I,1}(\eta;\phi_{j_1},\dots,\phi_{j_{p-1}}) \neq 0,
\]
and therefore $\Lambda \neq 0$.
This, and our earlier decision to avoid $\phi_j$ which solve certain linear equations, is why Theorem~\ref{T1.2} holds away from an algebraic set of codimension $1$.

In addition, to use Lemma~\ref{L6.1}, we need control of the Pr\"ufer phase $\theta(x)$. To get \eqref{6.6}, we took the imaginary part of \eqref{3.7}; to obtain information about the Pr\"ufer phase, we instead take the real part of \eqref{3.7} after the iterative process, so we have
\begin{equation}\label{6.10}
\frac{d\theta}{dx} \sim \Re \left( \Omega(x) +  \frac \Lambda{x^{(p-1)\gamma}} e^{i[\xi(x) + 2\theta(x)]}  \right).
\end{equation}
Here $\Omega(x)$ is the sum of terms
\begin{equation}\label{6.11}
\Omega(x) = \sum_{I=1}^{p-1} \sum_{\phi_{j_1}+\dots+\phi_{j_I}=0}
f_{I,0}(\eta;\phi_{j_1},\dots,\phi_{j_I}) \beta_{j_1}(x) \dots \beta_{j_I}(x)
\end{equation}
which we discarded in the proof of Theorem~\ref{T1.1} because it was real-valued, but for \eqref{6.10} we have to take it into account.

This is where the choice of $\beta_0(x)$ becomes important. Note that $\Omega(x)$ is a linear combination of functions of bounded variation, so $\Omega(x)$ has bounded variation; moreover, one of the terms in \eqref{6.11} is $-\frac 1\eta \beta_0(x)$, and all other terms are at least quadratic in the $\beta$'s,
\begin{equation}\label{6.12}
\Omega(x) = - \frac 1\eta \beta_0(x) + \mathcal{L}(\beta_0)(x)
\end{equation}
with
\begin{equation}\label{6.13}
 \mathcal{L}(\beta_0)(x) =  \sum_{I=2}^{p-1} \sum_{\phi_{j_1}+\dots+\phi_{j_I}=0}
f_{I,0}(\eta;\phi_{j_1},\dots,\phi_{j_I}) \beta_{j_1}(x) \dots \beta_{j_I}(x)
\end{equation}

If $\beta_0(x)$ weren't present in \eqref{6.13}, we could simply replace it by $\tilde \beta_0(x) = \beta_0(x) + \eta \Omega(x)$ and the new $\tilde\Omega(x)$ given by \eqref{6.12} would be $0$. Since $\beta_0(x)$ is present in \eqref{6.13}, destroying $\Omega(x)$ takes a little more work. Note that a priori we know that $\Omega(x) = O(x^{-\gamma})$ and $\beta_k(x) = O(x^{-\gamma})$, $\frac d{dx} \left( e^{-i\phi_k x} \beta_k(x) \right) = O(x^{-p\gamma})$ for all $k$.

\begin{lemma}
Let $\Omega(x)$ be given by \eqref{6.12}, \eqref{6.13} and $\Omega(x) = O(x^{-n\gamma})$, with $n\ge 1$. Replacing $\beta_0(x)$ by $\tilde\beta_0(x) = \beta_0(x) + \eta \Omega(x)$ on the right hand side of \eqref{6.11} leads to $\tilde \Omega = - \frac 1\eta \tilde \beta_0 + \mathcal{L}(\tilde \beta_0)$ with $\tilde\Omega(x) = O(x^{-(n+1)\gamma})$. If $\beta_0$ obeys the conditions \eqref{1.10}, then so does $\tilde \beta_0$.
\end{lemma}

\begin{proof}
Notice that
\[
\tilde\Omega = - \frac 1\eta \beta_0 - \Omega + \mathcal{L}(\tilde\beta_0) = \mathcal{L}(\tilde\beta_0) - \mathcal{L}(\beta_0)
\]
is, by \eqref{6.13}, a linear combination of products of $\Omega$ with one or more of the $\beta$'s and $\tilde\beta_0$; thus, since $\Omega(x) = O(x^{-n\gamma})$ and $\beta_k(x) = O(x^{-\gamma})$ for all $k$, we conclude $\tilde\Omega(x) = O(x^{-(n+1)\gamma})$. The claims about $\tilde \beta$ follow analogously.
\end{proof}

By applying this lemma $p-1$ times, we get from $\Omega(x) = O(x^{-\gamma})$ to $\Omega(x) = O(x^{-p\gamma}) \in L^1$, so \eqref{6.10} becomes
\begin{equation}\label{6.14}
\frac{d\theta}{dx} \sim \Re \left(  \frac \Lambda {x^{(p-1)\gamma}}e^{i[\xi(x) + 2\theta(x)]}  \right)
\end{equation}
With \eqref{6.6} and \eqref{6.14}, we are now ready to construct $\xi(x)$ which will lead to the desired asymptotics. Denote
\[
\psi(x) = \xi(x) + 2\theta(x).
\]

\begin{lemma}\label{L6.3}
Fix $E = \frac{\eta^2}4>0$ and let $R(x)$, $\theta(x)$ be the Pr\"ufer variables corresponding to some solution of \eqref{3.1}. Assume that \eqref{6.14}, \eqref{6.6} hold. Then we may pick $\xi(x)$ with $\xi'(x) \in O(x^{-(p-1)\gamma})$ such that
\begin{equation} \label{6.15}
\lim_{x\to\infty} \psi(x) = -\frac \pi 2 - \arg \Lambda.
\end{equation}
\end{lemma}

\begin{proof}
With $x_0$ to be specified later, pick $\xi(x)$ arbitrarily (e.g. constant) for $x<x_0$, and by the formula
\begin{equation}\label{6.16}
\frac d{dx} \xi(x) = - 2 \Re \left( \frac \Lambda {x^{(p-1)\gamma}} e^{i(\xi(x)+ 2\theta(x))} \right), \qquad x > x_0.
\end{equation}
Then $\xi'(x) \in O(x^{-(p-1)\gamma})$ is trivial.  By \eqref{6.14} and \eqref{6.16}, $\psi'(x) \sim 0$, so $\lim\limits_{x\to\infty} \psi(x)$ exists. The formula \eqref{6.16} determines $\xi(x)$ only up to a choice of initial condition $\xi(x_0)$. Alternatively, we can view this as a choice of initial condition  $\psi(x_0)$ for the function $\psi$. It remains to show that we can pick the value of the limit \eqref{6.15} by a suitable choice of $\xi(x_0)$.

The convergence of $\psi(x)$ followed, through \eqref{6.14}, from an iterative application of Lemma~\ref{L4.1}. Revisiting the proof of that lemma and assuming power law decay, we see that the same proof implies the following more quantitative version of the lemma: if $\Gamma(x)\in C^1$, $\lvert\frac d{dx} (e^{i\phi x}\Gamma(x)) \rvert \le C_1 x^{-p\gamma}$ and $\lim_{x\to\infty}\Gamma(x) = 0$, then
\[
\left\lvert \int_M^\infty \left( f(\eta)\Gamma(x) e^{ki[\eta x+2\theta(x)]} - g(\eta)\Gamma(x) e^{ki[\eta x+2\theta(x)]}   \frac {d\theta}{dx} \right) dx \right\rvert  \le 2 C_1 \lvert h(\eta)\rvert M^{1-p\gamma}
\]
Thus, under our current assumptions of power-law decay \eqref{1.8}, \eqref{6.9}, \eqref{6.16}, using this quantitative estimate for the rate of convergence,
\begin{equation}\label{6.17}
\lvert \psi(x) - \lim\limits_{x\to\infty} \psi(x) \rvert \le C x^{1-p\gamma},
\end{equation}
where $C$ depends only on $p$, the set of phases and $\Lambda$, but not on the choice of $\psi(x_0)$. Thus, convergence is uniform in different choices of this initial condition $\psi(x_0)$. Pick $x_0$ such that $\lvert \psi(x_0) - \lim\limits_{x\to\infty} \psi(x) \rvert < \pi$ for all initial values $\psi(x_0)$. Then, the map
 \[
 \exp(i\psi(x_0)) \mapsto \exp(i\lim\limits_{x\to\infty} \psi(x) )
 \]
  is a continuous (by uniform convergence) map from the unit circle to itself, which has no point $z$ which maps to its antipodal point $-z$. Thus, by standard topological considerations (see e.g.\ Hatcher~\cite[Section 2.2]{Hatcher02}), this map is homotopic to the identity map on the unit circle; further, since it isn't null-homotopic, it is onto (since the circle with one point removed has trivial fundamental group).
This implies that  \eqref{6.15} holds for some choice of $\psi(x_0)$ or, equivalently, $\xi(x_0)$.
\end{proof}

From now on, let us use the choice of $\xi(x)$ given by Lemma~\ref{L6.3}. From \eqref{6.17} and \eqref{6.14} it follows that
$
\frac {d\theta}{dx} \sim O(x^{-p\gamma}) \sim 0
$
so the limit
\begin{equation}
\theta_\infty = \lim\limits_{x\to\infty} \theta(x)
\end{equation}
exists. Similarly, \eqref{6.16} implies that $\xi(x)$ has bounded variation.

By \eqref{6.6}, \eqref{6.15} and \eqref{6.17}, we have
\begin{align*}
\frac d{dx} \log R(x) & \sim \Im \left( \frac \Lambda {x^{(p-1)\gamma}} e^{i\psi_\infty} + \frac \Lambda {x^{(p-1)\gamma}} (e^{i\psi(x)} - e^{i\psi_\infty} ) \right) \\
& \sim \Im \left( - i \frac{\lvert\Lambda\rvert}{x^{(p-1)\gamma}}  + O(x^{-p\gamma}) \right) \\
&  \sim - \frac{\lvert\Lambda\rvert}{x^{(p-1)\gamma}}
\end{align*}
Thus, Lemma~\ref{L6.1} is applicable and the asymptotics \eqref{6.2} hold, with $B=\lvert \Lambda\rvert$. This concludes the proof of Theorem~\ref{T1.2}.
\end{proof}

We only manipulated $\beta_0(x)$ in order to destroy $\Omega(x)$ from \eqref{6.10}, \eqref{6.11} (i.e.\ to make it $L^1$). There are other ways to do so, which do not involve $\beta_0(x)$. We illustrate this with the proof of Theorem~\ref{T1.3}.

\begin{proof}[Proof of Theorem~\ref{T1.3}] Since we are assuming $\beta_0(x)=0$, \eqref{6.11} becomes
\begin{equation*}
\Omega(x) = \frac 1{\eta} \sum_{k=1}^K \frac{\lambda_k^2}{\eta^2 - \alpha_k^2}\, \frac 1{x^{2\gamma}}.
\end{equation*}
Thus, if
\begin{equation*}
\min\{\alpha_1,\dots,\alpha_K\} < \eta <  \max\{\alpha_1,\dots,\alpha_K\},
\end{equation*}
we can choose $\lambda_1, \dots, \lambda_K>0$ so that
\[
\sum_{k=1}^K \frac{\lambda_k^2}{\eta^2 - \alpha_k^2} = 0
\]
and therefore $\Omega(x)=0$. The condition for this is a homogenous equation in $\lambda_1,\dots, \lambda_K$, so this choice of $\lambda_k$ does not hinder us in making a product of $\lambda$'s as large as wanted. With $\Omega(x)=0$, the remainder of the proof proceeds exactly as in the proof of Theorem~\ref{T1.2}.
\end{proof}

\bibliographystyle{amsplain}


\begin{thebibliography}{10}

\bibitem{Atkinson54}
F.~V. Atkinson, \emph{The asymptotic solution of second-order differential
  equations}, Ann. Mat. Pura Appl. (4) \textbf{37} (1954), 347--378.
  \MR{0067289 (16,701f)}

\bibitem{Behncke91}
H.~Behncke, \emph{Absolute continuity of {H}amiltonians with von {N}eumann
  {W}igner potentials. {II}}, Manuscripta Math. \textbf{71} (1991), no.~2,
  163--181. \MR{1101267 (93f:81031)}

\bibitem{Ben-ArtziDevinatz79}
Matania Ben-Artzi and Allen Devinatz, \emph{Spectral and scattering theory for
  the adiabatic oscillator and related potentials}, J. Math. Phys. \textbf{20}
  (1979), no.~4, 594--607. \MR{529723 (82a:35088a)}

\bibitem{GilbertPearson87}
D.~J. Gilbert and D.~B. Pearson, \emph{On subordinacy and analysis of the
  spectrum of one-dimensional {S}chr\"odinger operators}, J. Math. Anal. Appl.
  \textbf{128} (1987), no.~1, 30--56. \MR{915965 (89a:34033)}

\bibitem{HarrisLutz75}
W.~A. Harris, Jr. and D.~A. Lutz, \emph{Asymptotic integration of adiabatic
  oscillators}, J. Math. Anal. Appl. \textbf{51} (1975), 76--93. \MR{0369840
  (51 \#6069)}

\bibitem{Hatcher02}
Allen Hatcher, \emph{Algebraic topology}, Cambridge University Press,
  Cambridge, 2002. \MR{1867354 (2002k:55001)}

\bibitem{JanasSimonov10}
Jan Janas and Sergey Simonov, \emph{A {W}eyl-{T}itchmarsh type formula for a
  discrete {S}chr\"odinger operator with {W}igner--von {N}eumann potential},
  Studia Math. \textbf{201} (2010), no.~2, 167--189. \MR{2738159 (2011m:47059)}

\bibitem{KiselevLastSimon98}
Alexander Kiselev, Yoram Last, and Barry Simon, \emph{Modified {P}r\"ufer and
  {EFGP} transforms and the spectral analysis of one-dimensional
  {S}chr\"odinger operators}, Comm. Math. Phys. \textbf{194} (1998), no.~1,
  1--45. \MR{1628290 (99g:34167)}

\bibitem{Kruger12}
Helge Kr{\"u}ger, \emph{On the existence of embedded eigenvalues}, preprint.

\bibitem{Lukic1}
Milivoje Lukic, \emph{Orthogonal polynomials with recursion coefficients of
  generalized bounded variation}, Comm. Math. Phys. \textbf{306} (2011), no.~2,
  485--509. \MR{2824479 (2012f:42048)}

\bibitem{Naboko86}
S.~N. Naboko, \emph{On the dense point spectrum of {S}chr\"odinger and {D}irac
  operators}, Teoret. Mat. Fiz. \textbf{68} (1986), no.~1, 18--28. \MR{875178
  (88h:81029)}

\bibitem{Prufer26}
Heinz Pr{\"u}fer, \emph{Neue {H}erleitung der {S}turm-{L}iouvilleschen
  {R}eihenentwicklung stetiger {F}unktionen}, Math. Ann. \textbf{95} (1926),
  no.~1, 499--518. \MR{1512291}

\bibitem{ReedSimon2}
Michael Reed and Barry Simon, \emph{Methods of modern mathematical physics.
  {II}: {F}ourier analysis, self-adjointness}, Academic Press [Harcourt Brace
  Jovanovich Publishers], New York, 1975.

\bibitem{ReedSimon4}
\bysame, \emph{Methods of modern mathematical physics. {IV}. {A}nalysis of
  operators}, Academic Press [Harcourt Brace Jovanovich Publishers], New York,
  1978. \MR{0493421 (58 \#12429c)}

\bibitem{ReedSimon3}
\bysame, \emph{Methods of modern mathematical physics. {III}}, Academic Press
  [Harcourt Brace Jovanovich Publishers], New York, 1979, Scattering theory.
  \MR{529429 (80m:81085)}

\bibitem{Simon97}
Barry Simon, \emph{Some {S}chr\"odinger operators with dense point spectrum},
  Proc. Amer. Math. Soc. \textbf{125} (1997), no.~1, 203--208. \MR{1346989
  (97c:34179)}

\bibitem{Stolz92}
G{\"u}nter Stolz, \emph{Bounded solutions and absolute continuity of
  {S}turm-{L}iouville operators}, J. Math. Anal. Appl. \textbf{169} (1992),
  no.~1, 210--228. \MR{1180682 (93f:34141)}

\bibitem{Teschl09}
Gerald Teschl, \emph{Mathematical methods in quantum mechanics}, Graduate
  Studies in Mathematics, vol.~99, American Mathematical Society, Providence,
  RI, 2009, With applications to Schr{\"o}dinger operators. \MR{2499016
  (2010h:81002)}

\bibitem{WignerVonNeumann29}
John von Neumann and Eugene~P. Wigner, \emph{{\"U}ber merkw\"urdige diskrete
  {E}igenwerte}, Z. Phys. \textbf{30} (1929), 465--467.

\bibitem{Weidmann67}
Joachim Weidmann, \emph{Zur {S}pektraltheorie von
  {S}turm-{L}iouville-{O}peratoren}, Math. Z. \textbf{98} (1967), 268--302.
  \MR{0213915 (35 \#4769)}

\bibitem{Wong09}
Manwah~Lilian Wong, \emph{Generalized bounded variation and inserting point
  masses}, Constr. Approx. \textbf{30} (2009), no.~1, 1--15. \MR{2519651
  (2010f:42056)}

\end{thebibliography}

\providecommand{\bysame}{\leavevmode\hbox to3em{\hrulefill}\thinspace}
\providecommand{\MR}{\relax\ifhmode\unskip\space\fi MR }
\providecommand{\MRhref}[2]{%
  \href{http://www.ams.org/mathscinet-getitem?mr=#1}{#2}
}
\providecommand{\href}[2]{#2}

\end{document}